\documentclass[12pt,a4paper]{article}
\usepackage[top=2.5cm,bottom=2.5cm,left=2.5cm,right=2.5cm]{geometry}

\usepackage{amssymb,amsmath,graphicx,color,amsthm,centernot}
\usepackage{hyperref}
\usepackage[capitalise]{cleveref}
\usepackage[labelformat=simple]{subcaption}
\usepackage[font=small,labelfont=bf,width=12.5cm]{caption}

\usepackage{multirow,tabularx}

\newtheorem{theorem}{Theorem}[section]

\newtheorem{corollary}[theorem]{Corollary}
\newtheorem{proposition}[theorem]{Proposition}
\newtheorem{observation}[theorem]{Observation}
\newtheorem{conjecture}[theorem]{Conjecture}

\newcommand{\set}[1]{\ensuremath{\left\{#1 \right\}}}

    {\noindent \emph{Proof.} {}{#1}{}}{$~$\hfill $~\blacklozenge$ \vspace{0.2cm}}

\definecolor{defblue}{rgb}{0.4,0,0.84}
\definecolor{greyblue}{rgb}{0.23,0.4,0.70}
\definecolor{orange}{rgb}{1.0,0.5,0.2}
\definecolor{violet}{rgb}{0.55,0,0.55}

\newcommand{\cf}{\ensuremath{\chi_{\mathrm{iCFo}}}}
\newcommand{\pcf}{\ensuremath{\chi_{\mathrm{pCFo}}}}
\newcommand{\cfc}{\ensuremath{\chi_{\mathrm{iCFc}}}}
\newcommand{\pcfc}{\ensuremath{\chi_{\mathrm{pCFc}}}}
\newcommand{\um}{\ensuremath{\chi_{\mathrm{iUMo}}}}
\newcommand{\pum}{\ensuremath{\chi_{\mathrm{pUMo}}}}
\newcommand{\umc}{\ensuremath{\chi_{\mathrm{iUMc}}}}
\newcommand{\pumc}{\ensuremath{\chi_{\mathrm{pUMc}}}}

\usepackage{url}
\makeatletter
\g@addto@macro{\UrlBreaks}{\UrlOrds}
\makeatother

\newcolumntype{Y}{>{\centering\arraybackslash}X}

\begin{document}

\title{{\bf Proper conf\mbox{}lict-free and unique-maximum \\ colorings of planar graphs \\ with respect to neighborhoods}}

\author
{
	Igor Fabrici\thanks{Faculty of Science, Pavol Jozef \v Saf\'{a}rik University, Ko\v{s}ice, Slovakia. \newline
		E-Mails: \texttt{\{igor.fabrici,roman.sotak\}@upjs.sk, simona.rindosova@student.upjs.sk}}, \
	Borut Lu\v{z}ar\thanks{Faculty of Information Studies in Novo mesto, Slovenia. 
		E-Mail: \texttt{borut.luzar@gmail.com}},\thanks{University of Ljubljana, Faculty of Mathematics and Physics, Slovenia.} \
	Simona Rindo\v sov\'{a}\footnotemark[1], \
	Roman Sot\'{a}k\footnotemark[1]
}

\maketitle

{
\begin{abstract}
	A {\em conflict-free coloring} of a graph {\em with respect to open} (resp., {\em closed}) {\em neighborhood}
	is a coloring of vertices such that for every vertex there is a color appearing exactly once in its 
	open (resp., closed) neighborhood.
	Similarly, 
	a {\em unique-maximum coloring} of a graph {\em with respect to open} (resp., {\em closed}) {\em neighborhood}
	is a coloring of vertices such that for every vertex the maximum color appearing in its	open (resp., closed) neighborhood
	appears exactly once.
	
	In this paper, we study both colorings in the proper setting 
	(i.e., we require adjacent vertices to receive distinct colors),
	focusing mainly on planar graphs.
	Among other results, 
	we prove that every planar graph admits a proper unique-maximum coloring with respect to open neighborhood using at most $10$ colors,
	and give examples of planar graphs needing at least $6$ colors for such a coloring.		
	We also establish tight upper bounds for outerplanar graphs.	
\end{abstract}
}


\section{Introduction}

	A {\em conflict-free coloring} of a hypergraph is a coloring of the vertices
	such that in every hyperedge there is at least one color appearing only on one vertex.
	Motivated by the frequency assignment problem, 
	this type of coloring, in a language of set systems, was introduced by Even, Lotker, Ron, and Smorodinsky~\cite{EveLotRonSmo03} in 2003
	(see also~\cite{Smo03} for other applications),
	and received a considerable attention from the research community since then
	(for a survey see, e.g.,~\cite{Smo13} and the references therein).

	A related, but more restrictive notion is a {\em unique-maximum coloring} of a hypergraph;
	in such a coloring, in every hyperedge the maximal color appears on exactly one vertex.
	Clearly, every unique-maximum coloring is also conflict-free.
	For this reason and since it provides more structure, 
	the unique-maximum setting is often used also to prove results in the conflict-free setting~\cite{CheKesPal13}.

	The definition in terms of coloring hypergraphs provides a number of different variations of both colorings when restricted to graphs.
	In particular, the hyperedges of a hypergraph may represent, e.g., 
	neighborhoods of vertices of an underlying graph (see, e.g.,~\cite{AbeETAL17,BhyKalMat22,Che09,DebPrz22,KelRokSmo21,PacTar09}),
	paths in a graph (see, e.g.,~\cite{BorBudJenKra11,CheKesPal13,CheTot11,GreSkr12}),
	vertices incident with the same faces in a graph embedded in some surface (see, e.g.,~\cite{CzaJen09,FabGor16,Wen16}), etc.

	\medskip
	In this paper, we focus on conflict-free and unique-maximum colorings of graphs with respect to open and closed neighborhoods.	
	In these settings, the definitions read as follows.
	
	A {\em conflict-free coloring} of a graph $G$ {\em with respect to open} (resp., {\em closed}) {\em neighborhood}, 
	or an {\em $\mathrm{iCFo}$-coloring} (resp., an {\em $\mathrm{iCFc}$-coloring}) for short,
	is a coloring of the vertices such that in the open (resp., closed) neighborhood of every vertex
	there is at least one color appearing exactly on one vertex.
	The minimum number $k$ of colors such that $G$ admits an iCFo-coloring (resp., an iCFc-coloring) with $k$ colors is
	denoted by $\cf(G)$ (resp., $\cfc(G)$).
	
	A {\em unique-maximum coloring} of a graph $G$ {\em with respect to open} (resp., {\em closed}) {\em neighborhood}, 
	or an {\em $\mathrm{iUMo}$-coloring} (resp., an {\em $\mathrm{iUMc}$-coloring}) for short,
	is a coloring of the vertices such that in the open (resp., closed) neighborhood of every vertex
	the maximum color appears exactly on one vertex.
	The minimum number $k$ of colors such that $G$ admits an iUMo-coloring (resp., an iUMc-coloring) with $k$ colors is
	denoted by $\um(G)$ (resp., $\umc(G)$).
	
	There are many colorings regarding the neighborhoods of vertices under various assumptions,
	e.g., homogeneous colorings~\cite{JanMadSotLuz17} (in the open neighborhood of every vertex, the same number of colors appears), 
	dynamic~\cite{Mon01} (there is no vertex with only one color in its open neighborhood),
	adynamic colorings~\cite{SurLuzMad20} (there is at least one vertex with exactly one color in its open neighborhood), 
	odd colorings~\cite{PetSkr21} (in the open neighborhood of every vertex some color appears odd number of times),
	and square colorings~\cite{Tho18,Weg77} (in the open neighborhood of every vertex, every color appears at most once).	
	A vast majority of such colorings is also proper (i.e., adjacent vertices receive distinct colors)
	due to interesting combinatorial relationships with proper colorings. 	
		
	This motivated us to consider conflict-free and unique-maximum colorings with respect to neighborhood in a proper setting.
	We therefore define a {\em proper conflict-free coloring} of a graph $G$ {\em with respect to open} (resp., {\em closed}) {\em neighborhood}, 
	or a {\em $\mathrm{pCFo}$-coloring} (resp., a {\em $\mathrm{pCFc}$-coloring}) for short,
	as a proper coloring of the vertices such that in the open (resp., closed) neighborhood of every vertex
	there is at least one color appearing exactly on one vertex.
	The minimum number $k$ of colors such that $G$ admits a pCFo-coloring (resp., a pCFc-coloring) with $k$ colors is
	denoted by $\pcf(G)$ (resp., $\pcfc(G)$).
	Similarly, we define 
	a {\em proper unique-maximum coloring} of a graph $G$ {\em with respect to open} (resp., {\em closed}) {\em neighborhood}, 
	or a {\em $\mathrm{pUMo}$-coloring} (resp., a {\em $\mathrm{pUMc}$-coloring}) for short,
	as a proper coloring of the vertices such that in the open (resp., closed) neighborhood of every vertex
	the maximum color appears exactly on one vertex.
	The minimum number $k$ of colors such that $G$ admits a pUMo-coloring (resp., a pUMc-coloring) with $k$ colors is
	denoted by $\pum(G)$ (resp., $\pumc(G)$).

	\medskip
	This paper focuses on problems for the class of planar graphs,
	although we expect interesting properties of investigated colorings 
	will also be revealed for other classes and graphs in general.
	Restricting to planar graphs enables us to present results on four distinct proper colorings (conflict-free and unique-maximum regarding open and closed neighborhoods)
	while applying the same method of facial closures presented in the next section.
	Additionally, we discuss the corresponding results for the original (improper) four variants.	
	The conflict-free cases were already studied by Abel et al.~\cite{AbeETAL17}, Bhyravarapu and Kalyanasundaram~\cite{BhyKal20},
	Bhyravarapu, Kalyanasundaram, and Mathew~\cite{BhyKalMat22}, and Huang, Guo, and Yuan~\cite{HuaGuoYua20},
	whereas we introduce bounds for the unique-maximum cases.
	
	In particular, in~\cite{AbeETAL17}, partial conflict-free colorings were considered,
	meaning that some vertices may remain non-colored which (sometimes) results in a need for a new color in order to have all vertices colored.
	The authors proved a number of complexity results and several combinatorial bounds;
	namely, they proved that every $K_{k+1}$-minor-free graph $G$ has $\cfc(G) \le k+1$ ($k$ colors in a partial variant).
	It follows that every planar graph $G$ admits a partial coloring with at most $4$ colors. 
	In fact, as a corollary of the Four Color Theorem, we have $\cfc(G) \le 4$ (we discuss this in more detail in Observation~\ref{obs:pcfc}).
	They also proved that for every planar graph $G$, $\cf(G) \le 9$, that bound was improved to $6$ in~\cite{BhyKal20,BhyKalMat22},
	and to $5$ in~\cite{HuaGuoYua20}.
	In~\cite{BhyKal20,BhyKalMat22} and independently in~\cite{HuaGuoYua20}, it was also proved that for every outerplanar graph $G$, $\cf(G) \le 4$.
	
	We summarize known and our new results in Theorems~\ref{thm:im-pla}--\ref{thm:pl-out}, 
	and discuss them more thoroughly in subsequent sections.
	
	Note that for the open neighborhood variants of colorings presented in this paper, graphs with isolated vertices cannot be colored. 
	Therefore in those cases we restrict to graphs without isolated vertices, but we do not explicitly state it.
	
	Let $\mathcal{P}$ be the set of all planar graphs, 
	and let $\mathcal{O}$ be the set of all outerplanar graphs.
	Moreover, for an invariant $\chi_{.}$ and a graph class $\mathcal{C}$ define 
	$$
		\chi_{.}(\mathcal{C}) = \max \set{\chi_{.}(G) \ | \ G \in \mathcal{C}}\,.
	$$
	
	We first list results for improper variants on planar and outerplanar graphs.
	\begin{theorem}
		\label{thm:im-pla}
		For the class of planar graphs $\mathcal{P}$, it holds
		\begin{itemize}
			\item[$(a)$] $4 \le \cf(\mathcal{P}) \le 5$ (by~\cite{AbeETAL17} for the lower bound and by~\cite{HuaGuoYua20} for the upper bound);
			\item[$(b)$] $3 \le \cfc(\mathcal{P}) \le 4$ (by~\cite{AbeETAL17});
			\item[$(c)$] $5 \le \um(\mathcal{P}) \le 10$ (by Proposition~\ref{prop:iumo} and Theorem~\ref{thm:pumo});
			\item[$(d)$] $4 \le \umc(\mathcal{P}) \le 6$ (by Proposition~\ref{prop:iumc} and Theorem~\ref{thm:iumc}).
		\end{itemize}
	\end{theorem}	

	\begin{theorem}
		\label{thm:im-out}
		For the class of outerplanar graphs $\mathcal{O}$, it holds
		\begin{itemize}
			\item[$(a)$] $3 \le \cf(\mathcal{O}) \le 4$ (by~\cite{AbeETAL17} and~\cite{BhyKal20,BhyKalMat22,HuaGuoYua20});
			\item[$(b)$] $\cfc(\mathcal{O}) = 3$ (by~\cite{AbeETAL17});
			\item[$(c)$] $4 \le \um(\mathcal{O}) \le 5$ (by Proposition~\ref{prop:iumo-out} and Theorem~\ref{thm:pumo-out});
			\item[$(d)$] $3 \le \umc(\mathcal{O}) \le 5$ (by~\cite{AbeETAL17} and Corollary~\ref{cor:pumc-out}).
		\end{itemize}
	\end{theorem}

	The bounds for proper variants are next. 
	\begin{theorem}
		\label{thm:pr-pla}
		For the class of planar graphs $\mathcal{P}$, it holds
		\begin{itemize}
			\item[$(a)$] $6 \le \pcf(\mathcal{P}) \le 8$ (by Proposition~\ref{prop:pcfo} and Theorem~\ref{thm:pcfo} (and also as a corollary of~\cite[Theorem~28]{BhyKalMat22}));
			\item[$(b)$] $\pcfc(\mathcal{P}) = 4$ (by Observation~\ref{obs:pcfc});
			\item[$(c)$] $6 \le \pum(\mathcal{P}) \le 10$ (by Proposition~\ref{prop:pumo} and Theorem~\ref{thm:pumo});
			\item[$(d)$] $6 \le \pumc(\mathcal{P}) \le 8$ (by Proposition~\ref{prop:pumc} and Theorem~\ref{thm:pumc});
		\end{itemize}
	\end{theorem}

	Note that we establish tight bounds for outerplanar graphs for all four variants.
	\begin{theorem}
		\label{thm:pl-out}
		For the class of outerplanar graphs $\mathcal{O}$, it holds
		\begin{itemize}
			\item[$(a)$] $\pcf(\mathcal{O}) = 5$ (by Corollary~\ref{cor:pcfo-out} and Observation~\ref{obs:pcfo-out});
			\item[$(b)$] $\pcfc(\mathcal{O}) = 3$ (by Observation~\ref{obs:pcfc});
			\item[$(c)$] $\pum(\mathcal{O}) = 5$ (by Theorem~\ref{thm:pumo-out} and Observation~\ref{obs:pcfo-out});
			\item[$(d)$] $\pumc(\mathcal{O}) = 5$ (by Corollary~\ref{cor:pumc-out} and Proposition~\ref{prop:pumc-out});
		\end{itemize}
	\end{theorem}

	The rest of the paper is structured as follows.
	In Section~\ref{sec:prel}, we present terminology and introduce several auxiliary results.
	Then, in Sections~\ref{sec:improper-cf} and~\ref{sec:improper-um}, we discuss in more detail the bounds for improper variants.
	In the subsequent four sections, we prove results for proper variants, 
	and in Section~\ref{sec:con}, we present ideas for further work and propose several open problems.

\section{Preliminaries}
\label{sec:prel}

In this section, we present terminology, notation, and auxiliary results that we are using in our proofs.

The open (resp., closed) neighborhood of a vertex $v$ in a graph $G$ is denoted by $N_G(v)$ (resp., $N_G[v]$),
and we omit the graph reference if it is clear from the context. 
A vertex of degree $k$ (resp., at most $k$, at least $k$) is a {\em $k$-vertex} (resp., {\em $k^-$-vertex}, {\em $k^+$-vertex}).

In a conflict-free coloring, every vertex $v$ must have in its neighborhood a color, 
which appears only once; 
we call such a color {\em unique}. 
On the other hand, in a unique-maximum coloring, for every vertex $v$ 
in its neighborhood the maximum color appears only once; 
we call such a color {\em maximum} and denote it $\mu(v)$
(we again omit specifying the type of the neighborhood).

When coloring vertices with at most $k$ colors, we always use the colors from the set $\set{1,\dots,k}$,
and we omit this remark in the proofs. Also, when a coloring variant is clear from the context, 
we sometimes refer to it simply as a coloring.

We first establish an evident relationship between colorings with respect to open and closed neighborhoods.
\begin{proposition}	
	\label{prop:cl-op}
	For every graph $G$, it holds 
	$$
		\pcfc(G) \le \pcf(G) \quad\quad\quad \mathrm{and} \quad\quad\quad \pumc(G) \le \pum(G)\,.
	$$ 
\end{proposition}	

\begin{proof}
	Both statements follow from the fact that a pCFo-coloring (resp., pUMo-coloring) $\sigma$
	is also a pCFc-coloring (resp., pUMc-coloring).
	Consider a vertex $v$. 
	Since the coloring is proper, its color is distinct from all colors in its open neighborhood,
	and therefore it is distinct from its unique (resp., unique maximum) color.
	The conflict-free statement is thus established. For the unique-maximum, observe that 
	either $\sigma(v) < \mu(v)$ or $\sigma(v) > \mu(v)$ 
	(here, $\mu(v)$ is the unique-maximum for $v$ in the pUMo-coloring $\sigma$). 
	In the latter case, $\sigma(v)$ is the unique maximum color of $v$ in its closed neighborhood and we are done.
\end{proof}

Another straightforward relationship follows from the fact that every unique-maximum coloring is also conflict-free.
\begin{observation}
	For every graph $G$, it holds 
	\label{obs:cf-um}
	$$
		\pcfc(G) \le \pumc(G) \quad\quad\quad \mathrm{and} \quad\quad\quad \pcf(G) \le \pum(G)\,.
	$$
\end{observation}
It is also trivial to see that every proper coloring is also improper. 

An important concept which enabled us to considerably improve upper bounds on the number of colors 
is the following.
The {\em facial closure $\Phi_G(X)$} of a plane graph $G$ with respect to the set $X$ of vertices 
is the simple plane graph with the vertex set $V(\Phi_G(X)) = V(G) \setminus X$ 
and two vertices $u$ and $v$ are adjacent in $\Phi_G(X)$ if they are adjacent in $G$ 
or if there is a vertex $x \in V(G)$ such that $ux$ and $vx$ are consecutive edges on a boundary of a face in $G \setminus (X \setminus \set{x})$,
i.e., the graph $G$ with all vertices of $X$ except $x$ removed
(see Figure~\ref{fig:closure} for an illustration).
\begin{figure}[htp!]
	$$
	\includegraphics{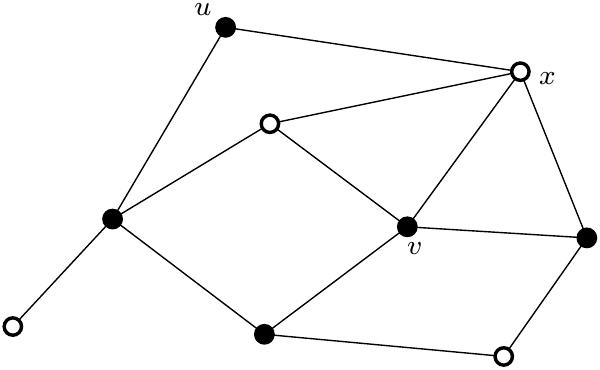} \quad\quad\quad\quad
	\includegraphics{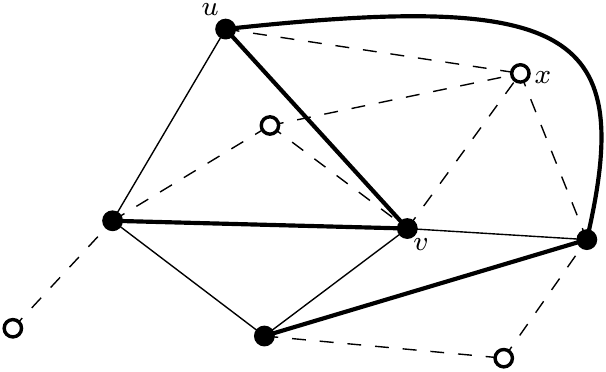}
	$$
	\caption{An example of a graph $G$ with the vertices of $X$ depicted as empty circles (left) 
		and the closure $\Phi_G(X)$ (right). 
		Note that the dashed edges and empty-circle vertices are not in $E(\Phi_G(X))$, and heavier edges are added.}
	\label{fig:closure}
\end{figure}
In other words, for every vertex $x \in X$ with at least $3$ neighbors from $V(G)\setminus X$,
there is a corresponding face in $\Phi_G(X)$ incident with all the neighbors of $x$ in $V(G)\setminus X$
in the order as they appear around $x$ in the embedding of $G$.
Additionally, if $x$ has $2$ neighbors in $V(G)\setminus X$, 
then the two neighbors are connected by an edge (we do not introduce parallel edges).
Note that a facial closure is also a plane graph.

In our proofs, we will use the following structural property of outerplanar graphs.
\begin{proposition}
	\label{prop:outer-2trian}
	In every outerplanar graph $G$ with minimum degree $2$, in which every $2$-vertex is adjacent to two $3^+$-vertices,
	there is a $2$-vertex incident with a triangle.
\end{proposition}

\begin{proof}
	Consider the block-tree $B$ of $G$, i.e., a tree in which every vertex corresponds to a block of $G$,
	and two vertices are adjacent in $B$ if the corresponding blocks have a common vertex in $G$.
	Let $x$ be a leaf in $B$ (or the only vertex of $B$ if $G$ is $2$-connected), 
	i.e., $x$ corresponds to a block $X$ of $G$ adjacent to at most one cutvertex.
	Note that the block $X$ is not a cycle, since there are no adjacent $2$-vertices in $G$.
	Therefore, the weak-dual of $X$ is a tree with at least $2$ vertices, and thus with at least two leaves.
	Recall that a leaf in a weak dual corresponds to a face $f_0$ with only one edge on the boundary which is not
	incident with the outerface. 
	Consequently, the face $f_0$ is a $3$-cycle in $X$ incident with exactly one $2$-vertex of $X$. 
	Since there are at least two such faces, at most one of them is incident with a $2$-vertex which is not a cutvertex of $G$, 
	and thus there is a $3$-cycle incident with a $2$-vertex in $G$.
\end{proof}

In some of our proofs, we are using results on conflict-free colorings of plane graphs with respect to faces.
In particular, a {\em facial conflict-free coloring} of a plane graph $G = (V,E,F)$ 
is a proper coloring of its vertices such that for every face $f \in F(G)$ there exists a color
which appears exactly once on the vertices incident with $f$.
A tight upper bound for this coloring was established by Czap and Jendrol'~\cite[Theorem~3.3]{CzaJen09}.
We will use its strengthened version.
\begin{theorem}
	\label{thm:fac_cf}
	For every plane graph $G$ and 
	every function $\zeta$ choosing for every face $f \in F(G)$ a vertex $\zeta(f)$ incident with $f$,
	there is a proper facial conflict-free coloring with at most $4$ colors
	such that the color of $\zeta(f)$ is a unique color for $f$.
\end{theorem}

\begin{proof}
	Let $G'$ be the graph obtained from $G$ by connecting every vertex $\zeta(f)$, for every $f \in F(G)$,
	with the vertices incident with $f$ which are not yet adjacent to $\zeta(f)$.
	By the Four Color Theorem, there is a proper coloring $\sigma$ of $G'$ using at most $4$ colors.
	Since for every face $f \in F(G)$, all its incident vertices are adjacent to $\zeta(f)$ in $G'$
	and hence have distinct colors, it follows that $\sigma$ is a proper facial conflict-free coloring of $G$.
\end{proof}

Similarly, a {\em facial unique-maximum coloring} of a plane graph $G = (V,E,F)$ 
is a proper coloring of its vertices such that for every face $f \in F(G)$ the maximum color
on the vertices incident with $f$ appears exactly once. 
Wendland~\cite{Wen16} improved the bound of $6$ colors from~\cite{FabGor16} to $5$.
Later, it turned out that there is an infinite family of examples attaining it~\cite{LidMesSkr18}, and so 
the upper bound is tight.
\begin{theorem}[Wendland~\cite{Wen16}]
	\label{thm:fac_um}
	Every plane graph $G$ admits a proper facial unique-maximum coloring with at most $5$ colors.
\end{theorem}

\section{Improper conflict-free colorings}
\label{sec:improper-cf}

In this section, we review results on both coloring variants in the improper setting.

The upper bounds for iCFc-coloring of planar and outerplanar graphs were already established 
in~\cite{AbeETAL17} (see also Observation~\ref{obs:pcfc}).
The bound of $3$ colors for outerplanar (and planar) graphs is attained, e.g., by $C_5$ with one diagonal.
Abel et al.~\cite[Lemma~3.2]{AbeETAL17} also provided a more general construction of graphs $G_k$
for which $\cfc(G_k) \ge k$ (see Figure~\ref{fig:abelG3} for an example).
\begin{figure}[htp!]
	\centering
	\includegraphics{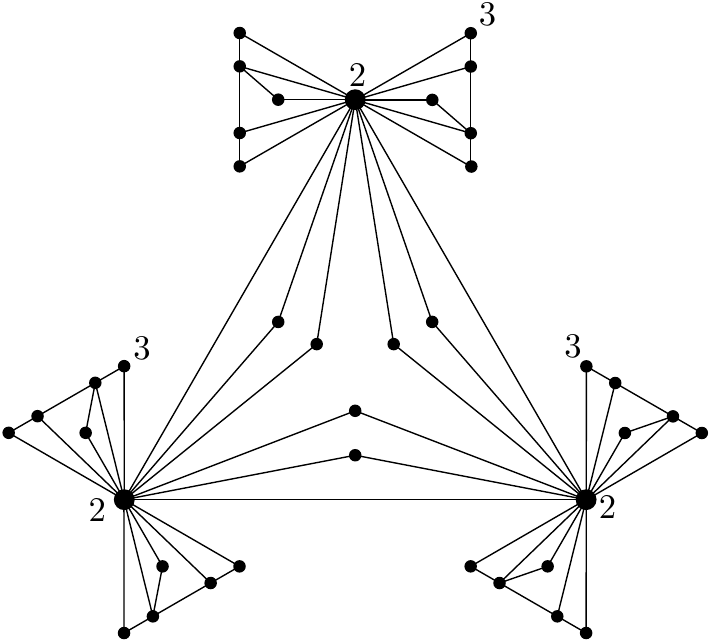}
	\caption{The planar graph $G_3$ with $\cfc(G_3) = 3$. Vertices with no color assigned are colored with color $1$.}
	\label{fig:abelG3}
\end{figure}	

The case of iCFo-coloring is even a bit more interesting.
The lower bounds are established by the fact, observed in~\cite{AbeETAL17},
that $\cf(\mathcal{S}(G)) \ge \chi(G)$, where $\mathcal{S}(G)$ is the graph obtained from $G$ by subdividing every edge once (see Figure~\ref{fig:icfo-noneq}).
\begin{figure}[htp!]
	\centering
	\includegraphics{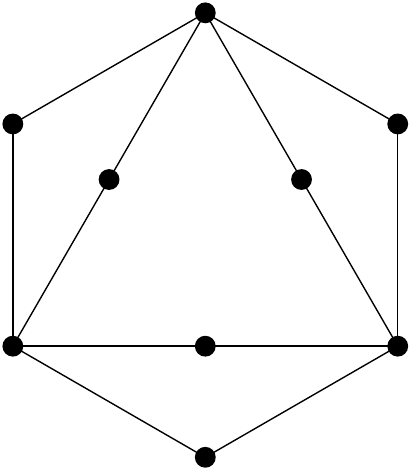}\quad\quad\quad
	\includegraphics{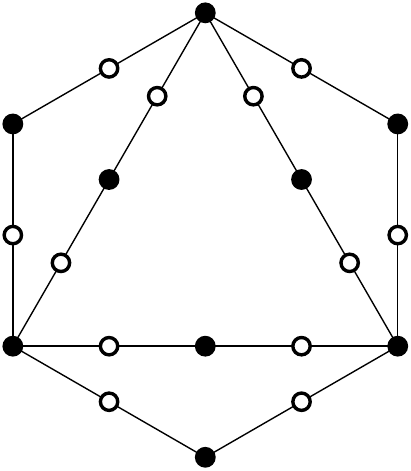}
	\caption{A planar graph $G$ with $\chi(G)=2$ (left) 
		and its subdivision $\mathcal{S}(G)$ with $\cf(\mathcal{S}(G))=3$ (right).}
	\label{fig:icfo-noneq}
\end{figure}
Every new vertex, added by subdivision, needs to have both neighbors colored with distinct colors
and thus at least $\chi(G)$ colors are needed for an iCFo-coloring of $\mathcal{S}(G)$.
Let us remark also that the equality does not hold for all graphs; 
consider, e.g., the graph depicted in Figure~\ref{fig:icfo-noneq}.

The upper bounds $5$ and $4$ for $\cf(\mathcal{P})$ and $\cf(\mathcal{O})$, respectively, were proved in~\cite{HuaGuoYua20} and~\cite{BhyKal20,BhyKalMat22,HuaGuoYua20},
and it seems there is still some space for their improvement.

\section{Improper unique-maximum colorings}
\label{sec:improper-um}

We are not aware of any specific results on unique-maximum colorings of planar graphs with respect to neighborhoods.
Therefore, the only known bounds are the constructions for lower bounds on conflict-free variants by Observation~\ref{obs:cf-um}.
We present here constructions improving all the lower bounds, except for the case $(d)$ of Theorem~\ref{thm:im-out}.

For the lower bound in Theorem~\ref{thm:im-pla}$(d)$, we use a bit simplified construction of Abel et al. mentioned in the previous section.
\begin{proposition}
	\label{prop:iumc}
	The planar graph $G_3'$ has $\umc(G_3') = 4$.
\end{proposition}

\begin{figure}[htp!]
	\centering
	\includegraphics{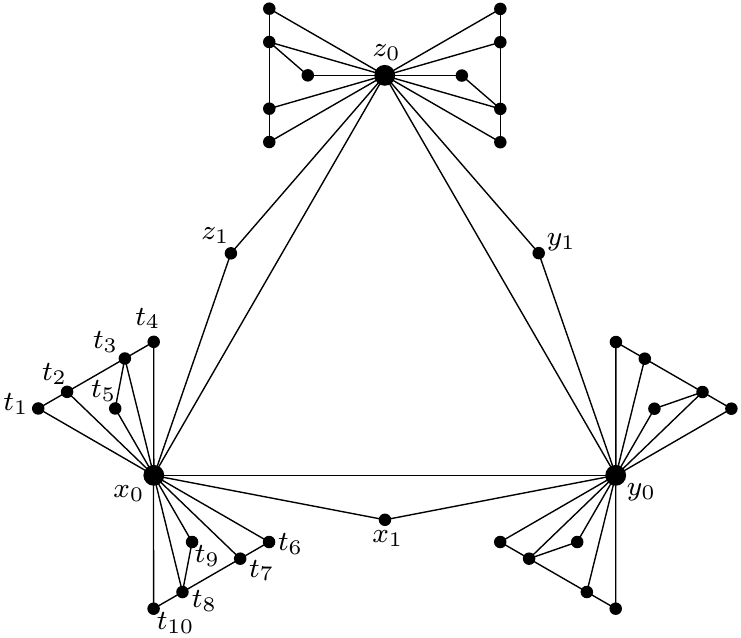}
	\caption{The planar graph $G_3'$ with $\umc(G_3') = 4$.}
	\label{fig:abelG3-um}
\end{figure}	

\begin{proof}
	Suppose to the contrary that there is an iUMc-coloring $\sigma$ of $G_3'$ using at most $3$ colors.
	We consider the cases regarding the color of $x_0$ and we use the vertex labelings as depicted in Figure~\ref{fig:abelG3-um}.
	
	Suppose first that $\sigma(x_0) = 1$.
	Then at most one of the vertices $t_i$, $i \in \set{1,\dots,10}$, is colored with $3$.
	This means that we may assume, without loss of generality, 
	that the vertices $t_j$, $j \in \set{6,\dots,10}$ are colored with $1$ or $2$.
	Since each vertex of color $1$ needs to be adjacent to exactly one vertex of color $2$ while vertices of color $2$ not being adjacent, this is not possible.
	
	Thus, by symmetry, none of the vertices $x_0$, $y_0$, and $z_0$ is colored with $1$, and so they must be colored with $2$ or $3$. 
	Moreover, at most one of them has color $3$.
	Thus, we may assume that $\sigma(x_0) = \sigma(y_0) = 2$ and consequently, $\sigma(x_1) = 3$.
	But then $\sigma(z_0) = 2$ and $\sigma(z_1) \le 2$, a contradiction.
	
	On the other hand, $4$ colors suffice; e.g., color $x_0$, $y_0$, and $z_0$ with $2$, $3$, and $4$, respectively,
	and all other vertices with $1$.
\end{proof}

The lower bound in Theorem~\ref{thm:im-out}$(c)$ is attained by the graph in Figure~\ref{fig:iumo-out}.
\begin{proposition}
	\label{prop:iumo-out}
	The outerplanar graph $O_{\mathrm{iUMo}}$ has $\um(O_{\mathrm{iUMo}})=4$.
\end{proposition}
\begin{figure}[htp!]
	\centering
	\includegraphics{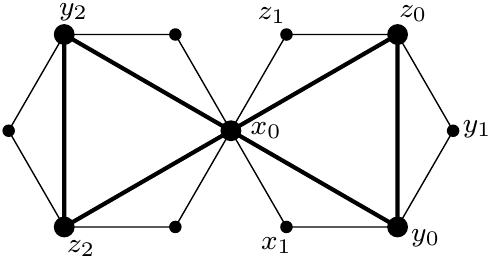}
	\caption{The outerplanar graph $O_{\mathrm{iUMo}}$ with $\um(O_{\mathrm{iUMo}}) = 4$.}
	\label{fig:iumo-out}
\end{figure}	

\begin{proof}
	Suppose to the contrary that there is an iUMo-coloring $\sigma$ of $O_{\mathrm{iUMo}}$ using at most $3$ colors.
	Note that $x_0$, $y_0$, and $z_0$ are colored distinctly, 
	otherwise at least one of the vertices $x_1$, $y_1$, and $z_1$ does not have a unique color in its open neighborhood.
	By the same argument as above, we infer that $\set{\sigma(x_0),\sigma(y_2),\sigma(z_2)} = \set{1,2,3}$, 
	and so either $\sigma(x_0) < 3$ and $x_0$ is adjacent to two vertices of the maximum color $3$,
	or $\sigma(x_0) = 3$. 
	In the latter case, we have two subcases. 
	First, $\mu(x_0) = 3$ and color $3$ appears on an adjacent $2$-vertex $v$,
	in which case the other neighbor of $v$ is incident with two vertices of the maximum color $3$. 
	Second, $\mu(x_0) < 3$, but then $x_0$ is adjacent to at least two vertices of color $2$, a contradiction.
\end{proof}

We use a construction analogous to the one in the previous case to obtain the lower bound in Theorem~\ref{thm:im-pla}$(c)$; 
see the graph in Figure~\ref{fig:iumo-pla}.
\begin{proposition}
	\label{prop:iumo}
	The planar graph $H_{\mathrm{iUMo}}$ has $\um(H_{\mathrm{iUMo}}) = 5$.
\end{proposition}
\begin{figure}[htp!]
	\centering
	\includegraphics{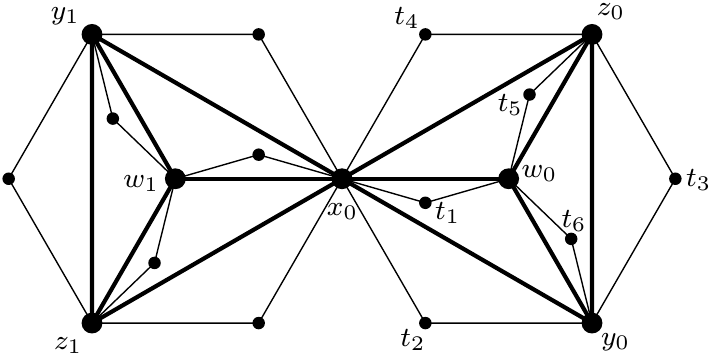}
	\caption{The planar graph $H_{\mathrm{iUMo}}$ with $\um(H_{\mathrm{iUMo}}) = 5$.}
	\label{fig:iumo-pla}
\end{figure}	

\begin{proof}
	Suppose to the contrary that there is an iUMo-coloring $\sigma$ of $H_{\mathrm{iUMo}}$ using at most $4$ colors.
	Note that $x_0$, $y_0$, $w_0$ and $z_0$ are colored distinctly, 
	otherwise at least one of the vertices $t_i$, $1 \le i \le 6$, does not have a unique color in its open neighborhood.
	By the same argument as above, we infer that $\set{\sigma(x_0),\sigma(y_1),\sigma(w_1),\sigma(z_1)} = \set{1,2,3,4}$, 
	and so either $\sigma(x_0) < 4$ and $x_0$ is adjacent to two vertices of the maximum color $4$, 
	or $\sigma(x_0) = 4$.
	In the latter case, we have two subcases. 
	First, $\mu(x_0) = 4$ and color $4$ appears on an adjacent $2$-vertex $v$,
	in which case the other neighbor of $v$ is incident with two vertices of the maximum color $4$. 
	Second, $\mu(x_0) < 4$, but then $x_0$ is adjacent to at least two vertices of color $3$, a contradiction.	
\end{proof}

Regarding the upper bounds, we were only able to establish one result better than the one from the proper setting;
namely, we prove an upper bound for the iUMc-coloring of planar graphs.
\begin{theorem}
	\label{thm:iumc}
	For every planar graph $G$, it holds
	$$
		\umc(G) \le 6\,.
	$$	
\end{theorem}

\begin{proof}
	By abusing the notation, we let $G$ represent also a fixed plane embedding of $G$.
	Color the vertices of $G$ properly with positive integers such that the lowest possible color is always assigned to a current vertex.
	In this way, we obtain a coloring $\sigma$ in which every vertex is either colored with $1$ or it is adjacent to a vertex of color $1$.
	
	Let $V_1$ be the set of vertices of $G$ colored by $1$ and $V_2 = V(G)\setminus V_1$.	
	Now, we will use the facial closure $\Phi_G(V_2)$.
	By Theorem~\ref{thm:fac_um}, it admits a facial unique-maximum coloring $\alpha$ using at most $5$ colors;
	we use the set of colors $\set{2,3,4,5,6}$.
	
	Next, recolor the vertices of $G$ by setting
	$\sigma(v) = \alpha(v)$ for every $v \in V_1$
	and $\sigma(v) = 1$ for every $v \in V_2$.
	We claim that the coloring $\sigma$ is an iUMc-coloring of $G$.	
	Indeed, for every vertex $v$ in $V_2$,
	there is a unique maximal color from $\set{2,\dots,6}$ in its neighborhood, 
	since it is adjacent to at least one vertex from $V_1$;
	The cases with $d(v) \le 2$ are trivial, and the case $d(v) \ge 3$ follows from the properties of $\Phi_G(V_2)$.
	For every vertex $v$ in $V_1$, all its neighbors have color $1$, and thus $\mu(v) = \sigma(v)$.
\end{proof}

\section{Proper conflict-free coloring with respect to open neighborhood}

We begin by discussing outerplanar graphs.
By Theorem~\ref{thm:pumo-out} and Observation~\ref{obs:cf-um}, we have the following.
\begin{corollary}
	\label{cor:pcfo-out}
	For every outerplanar graph $G$, it holds
	$$
		\pcf(G) \le 5\,.
	$$
\end{corollary}

The tightness of the bound is realized by the cycle $C_5$, 
since the two neighbors of every vertex must have distinct colors,
and the coloring must be proper. 
It means that we are coloring the square of $C_5$, which is $K_5$, and hence $\pcf(C_5) = 5$.
An analogous argument gives the following.
\begin{observation}
	\label{obs:pcfo-out}
	For every integer $n$, $n \ge 3$, it holds
	$$
		\pcf(C_n) \le 5
	$$
	and the upper bound is achieved only in the case $n = 5$.
\end{observation}

We continue with consideration of planar graphs.	
\begin{theorem}
	\label{thm:pcfo}
	For every planar graph $G$, it holds
	$$
		\pcf(G) \le 8\,.
	$$
\end{theorem}

\begin{proof}
	By abusing the notation, we let $G$ represent also a fixed plane embedding of $G$.
	Color the vertices of $G$ properly with positive integers such that the lowest possible color is always assigned to a current vertex.
	In this way, we obtain a coloring in which every vertex is either colored with $1$ or it is adjacent to a vertex of color $1$.
	
	Let $V_1$ be the set of vertices of $G$ colored by $1$ and $V_2 = V(G)\setminus V_1$.	
	We will use the facial closures $\Phi_G(V_1)$ and $\Phi_G(V_2)$.
	By Theorem~\ref{thm:fac_cf}, they admit facial conflict-free colorings $\alpha_1$ and $\alpha_2$, respectively, using at most $4$ colors;
	we use two distinct sets of at most $4$ colors for them, 
	say, $\set{1,2,3,4}$ for $\Phi_G(V_1)$ and $\set{5,6,7,8}$ for $\Phi_G(V_2)$.
	
	Now, recolor the vertices of $G$ to obtain a coloring $\sigma$ by setting
	$\sigma(v) = \alpha_2(v)$ for every $v \in V_1$ and
	$\sigma(v) = \alpha_1(v)$ for every $v \in V_2$.
	We claim that the coloring $\sigma$ is a pCFo-coloring of $G$.
	Indeed, since the colorings $\alpha_1$ and $\alpha_2$ use distinct colors, 
	and they preserve adjacencies within the sets $V_1$ and $V_2$, it follows that $\sigma$ is a proper coloring.
	Next, for every vertex $v$ in $V_i$, $i \in \set{1,2}$,
	there is a unique color in its neighborhood;
	namely, 
	if $d(v) = 1$, then there is only one color in $N(v)$,
	if $d(v) = 2$, then the two neighbors are either adjacent in $\Phi_G(V_{3-i})$ 
	or already in $G$, in both cases they receive distinct colors,
	and finally, 
	if $d(v) \ge 3$, then there is a unique color in $N(v)$ representing a unique color for the face in $\Phi_G(V_{3-i})$
	comprised by the neighbors of $v$.
\end{proof}

Let us mention here that in the proof we use an idea similar to the one used by Abel et al.~\cite{AbeETAL17}
to prove an existence of an iCFo-coloring of planar graphs using at most $8$ colors. 
The crucial difference is in using the notion of facial closure and results for the facial version of conflict-free coloring.
We also note that the same bound of $8$ colors follows from the proof of Theorem~28 in~\cite{BhyKalMat22}, 
by applying the Four Color Theorem to the vertices of $V_2$.

We do not believe that the upper bound of $8$ colors is tight. 
But, as the example below shows, there are planar graphs that need at least $6$ colors for a pCFo-coloring.

\begin{proposition}
	\label{prop:pcfo}
	The planar graph $H_{\mathrm{pCFo}}$ has $\pcf(H_{\mathrm{pCFo}}) = 6$.
\end{proposition}
	
\begin{figure}[htp!]
	\centering
	\includegraphics{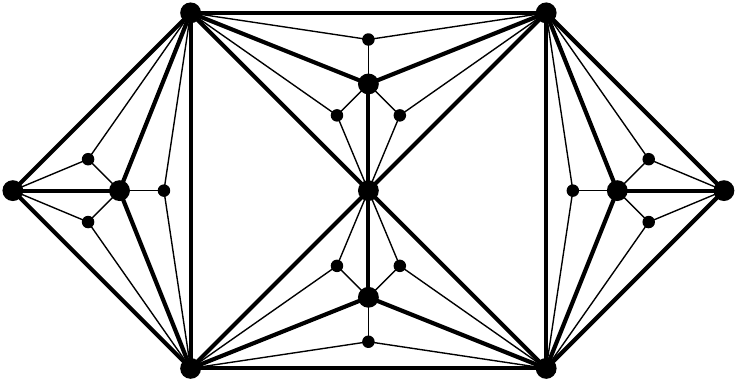}
	\caption{The planar graph $H_{\mathrm{pCFo}}$ with $\pcf(H_{\mathrm{pCFo}}) = 6$.}
	\label{fig:pcfo6}
\end{figure}
	
\begin{proof}
	Suppose to the contrary that $H_{\mathrm{pCFo}}$ admits a pCFo-coloring $\sigma$ with at most $5$ colors.
	Consider first the configuration $H$ depicted in Figure~\ref{fig:pcfo6-sub}.
	\begin{figure}[htp!]
		\centering
		\includegraphics{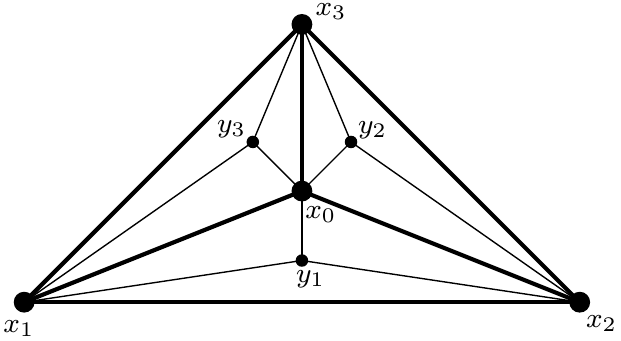}
		\caption{A subgraph $H$ of $H_{\mathrm{pCFo}}$.}
		\label{fig:pcfo6-sub}
	\end{figure}	
	The vertices $x_0$, $x_1$, $x_2$, and $x_3$ receive $4$ distinct colors, say $1$, $2$, $3$, and $4$, respectively.
	It follows that at least one of the vertices $y_1$, $y_2$, and $y_3$ is colored with $5$, 
	otherwise $x_0$ does not have a unique color in its neighborhood.
	This means that at least two of the vertices $x_1$, $x_2$, and $x_3$ have $4$ 
	distinct colors in their open neighborhoods.

	Now, note that $H_{\mathrm{pCFo}}$ is comprised of four copies of $H$,
	with five vertices, not corresponding to $x_0$, identified in such a way
	that every identified vertex belongs to two copies of $H$.
	Altogether, in $H_{\mathrm{pCFo}}$ there are seven vertices corresponding to at least one of $x_1$, $x_2$, and $x_3$.
	Since in each copy of $H$ at least two of them see $4$ distinct colors in their open neighborhoods,
	at least one of the vertices sees all $4$ colors at least twice, a contradiction.	
	
	It remains to prove that $6$ colors are indeed sufficient.
	We leave this verification to the reader.
\end{proof}

\section{Proper conflict-free coloring with respect to closed neighborhood}

The proper conflict-free coloring of a graph $G$ with respect to closed neighborhood is equivalent 
to a proper coloring of $G$, since each vertex is colored with a color different from the 
colors on the neighbors of the vertex.
\begin{observation}
	\label{obs:pcfc}
	For every graph $G$, it holds
	$$
		\pcfc(G) = \chi(G)\,.
	$$
\end{observation}

Planar graphs thus require at most $4$ colors by the Four Color Theorem with, e.g., 
$K_4$ having $\pcfc(K_4) = 4$. 
Outerplanar graphs are $2$-degenerate and hence properly $3$-colorable with, e.g., 
$C_3$ having $\pcfc(C_3) = 3$.

\section{Unique-maximum coloring of plane graphs with respect to open neighborhood}

The unique-maximum colorings require a more rigorous approach.
We again begin with bounds for outerplanar graphs.
\begin{theorem}
	\label{thm:pumo-out}
	For every outerplanar graph $G$, it holds
	$$
		\pum(G) \le 5\,.
	$$
\end{theorem}

\begin{proof}
	We prove the theorem by contradiction.
	Suppose that $G$ is an outerplanar graph with the minimum number of vertices 
	which does not admit a pUMo-coloring with at most $5$ colors.
	By the $2$-degeneracy of outerplanar graphs, there is a $2^{-}$-vertex $v$ in $G$.	
	
	In the case of $d(v) = 1$,	
	consider the graph $G' = G \setminus \set{v}$.
	By the minimality, there is a pUMo-coloring $\sigma$ of $G'$ with at most $5$ colors.	
	The coloring $\sigma$ induces a partial pUMo-coloring of $G$ with only $v$ being non-colored.	
	Let $v_1$ be the neighbor of $v$.
	We color $v$ with a color distinct $\sigma(v_1)$ and $\mu(v_1)$.
	We have at least $3$ available colors, and coloring $v$ with any of them completes $\sigma$, 
	since the coloring is clearly proper and $v_1$ has a unique-maximum: either the color of $v$ or the unique-maximum in $G'$.
	
	Hence, we may assume that the minimum degree of $G$ is $2$.
	Suppose now that there are two adjacent $2$-vertices $v$ and $w$ in $G$.
	Consider the graph $G' = G \setminus \set{v,w}$ and its pUMo-coloring $\sigma$ of $G'$ with at most $5$ colors.
	Again, $\sigma$ induces a partial pUMo-coloring of $G$, now with $v$ and $w$ being non-colored.
	Let $v_1$ be the neighbor of $v$ distinct from $w$, 
	and $w_1$ be the neighbor of $w$ distinct from $v$.	
	We color $v$ with a color distinct from $\sigma(v_1)$, $\mu(v_1)$, and $\sigma(w_1)$.
	There are at least two such colors.
	Then, we color $w$ with a color distinct from $\sigma(v)$, $\sigma(v_1)$, $\sigma(w_1)$, and $\mu(w_1)$.
	There is at least one such color.
	Regardless of the choice, $v$ will always have a unique maximum color in its open neighborhood, 
	since the two colors in it are distinct, 
	and similarly $w$ has only distinct colors in its open neighborhood.
	Note that also the vertices $v_1$ and $w_1$ either retain their original unique maximum colors,
	or the color of $v$ and $w$ becomes the new unique maximum color, respectively.	
	
	Therefore, there are no adjacent $2$-vertices in $G$.
	By Proposition~\ref{prop:outer-2trian}, there is a $2$-vertex $v$ incident with a $3$-cycle $vv_1v_2$ in $G$.	
	Consider the graph $G' = G \setminus \set{v}$ and its pUMo-coloring $\sigma$ of $G'$ with at most $5$ colors.
	As above, $\sigma$ induces a partial pUMo-coloring of $G$, now with $v$ being non-colored.
	We color $v$ with a color distinct from $\sigma(v_1)$, $\sigma(v_2)$, $\mu(v_1)$, and $\mu(v_2)$.
	There is at least one such color, and regardless of the choice, 
	$v$ will always have a unique maximum color in its open neighborhood, 
	since both colors in it are distinct.
	Moreover, $v_1$ and $v_2$ either retain their original unique maximum colors,
	or the color of $v$ becomes their new unique maximum color.	
\end{proof}	

The upper bound of $5$ colors for outerplanar graphs is tight, e.g., $\pum(C_5) = 5$; 
it follows from Observations~\ref{obs:cf-um} and~\ref{obs:pcfo-out}.

We continue by establishing bounds for planar graphs.
\begin{theorem}
	\label{thm:pumo}
	For every planar graph $G$, it holds
	$$
		\pum(G) \le 10\,.
	$$
\end{theorem}
	
\begin{proof}
	By abusing the notation, we let $G$ represent also a fixed plane embedding of $G$.
	Color the vertices of $G$ properly with positive integers such that the lowest possible color is always assigned to a current vertex.
	In this way, we obtain a coloring in which every vertex is either colored with $1$ or it is adjacent to a vertex of color $1$.
	
	Let $V_1$ be the set of vertices of $G$ colored by $1$ and $V_2 = V(G)\setminus V_1$.	
	Now, we will use the facial closures $\Phi_G(V_1)$ and $\Phi_G(V_2)$.
	By Theorem~\ref{thm:fac_um}, they admit facial unique-maximum colorings $\alpha_1$ and $\alpha_2$, respectively, using at most $5$ colors;
	we use two distinct sets of at most $5$ colors for them; 
	namely, $\set{1,2,3,4,5}$ for $\Phi_G(V_1)$ and $\set{6,7,8,9,10}$ for $\Phi_G(V_2)$.
	
	Now, recolor the vertices of $G$ to obtain a coloring $\sigma$ by setting
	$\sigma(v) = \alpha_2(v)$ for every $v \in V_1$ and
	$\sigma(v) = \alpha_1(v)$ for every $v \in V_2$.
	We claim that the coloring $\sigma$ is a proper unique-maximum coloring of $G$.
	Indeed, since the colorings $\alpha_1$ and $\alpha_2$ use distinct colors, 
	and they preserve adjacencies within the sets $V_1$ and $V_2$, it follows that $\sigma$ is a proper coloring.
	Next, for every vertex $v$ in $V_2$,
	there is a unique maximal color from $\set{6,\dots,10}$ in its neighborhood, 
	since it is adjacent to at least one vertex from $V_1$;
	The cases with $d(v) \le 2$ are again trivial, and the case $d(v) \ge 3$ follows from the properties of $\Phi_G(V_2)$.
	For every vertex $v$ in $V_1$, there is no color from $\set{6,\dots,10}$ in its neighborhood 
	and therefore, its unique maximal color is guaranteed by the properties of $\Phi_G(V_1)$.
\end{proof}

The lower bound for the pUMo-chromatic number is at least $6$ as follows already from Proposition~\ref{prop:pcfo} and Observation~\ref{obs:cf-um}.
However, we present another, well-known graph on just $9$ vertices that attains this bound.

\begin{proposition}
	\label{prop:pumo}
	The Fritsch graph $G$ has $\pum(G) = 6$.
\end{proposition}

\begin{figure}[htp!]
	$$
		\includegraphics{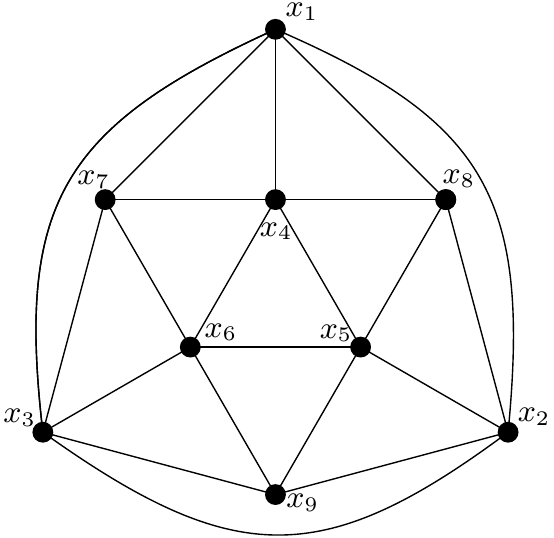}
	$$
	\caption{The Fritsch $G$ graph with $\pum(G) = 6$.}
	\label{fig:fritsch}
\end{figure}

\begin{proof}
	Let $G$ be the Fritsch graph depicted in Figure~\ref{fig:fritsch}.
	Suppose to the contrary that $G$ has a pUMo-coloring $\sigma$ with $5$ colors.
	
	First, since $G$ has diameter $2$, there is at most one vertex of color $5$.
	Moreover, we may assume that there is a vertex $x$ such that $\sigma(x) = 5$,
	otherwise we just omit one of the lower colors.	
	Let $\set{a,b,c,d} = \set{1,2,3,4}$.
	Note that due to the symmetry, there are two possibilities for $x$.
	We consider the two cases separately.	
	
	\medskip
	\noindent \textbf{Case 1:}	
	Suppose first that $x = x_1$, i.e., $\sigma(x_1) = 5$. \quad
	Let $\sigma(x_5) = a$, $\sigma(x_6) = b$, and $\sigma(x_9) = c$.
	Then $\sigma(x_2) \in \set{b,d}$ and $\sigma(x_3) \in \set{a,d}$.
	Since $\sigma(x_2) \neq \sigma(x_3)$ and by symmetry, we may assume, without loss of generality, that $\sigma(x_3) = a$.
	It follows that $\sigma(x_2) = d$, otherwise $x_9$ does not have a unique maximum color in its open neighborhood.
	Moreover, $a < \mu(x_9) \in \set{b,d}$.
	Next, observe that $\set{\sigma(x_4), \sigma(x_7)} = \set{c,d}$.
	Consequently, $\mu(x_6) = d$ and $d > a$, $d > c$.
	But now, since $d$ appears twice in the open neighborhood of $x_1$, we have that $\mu(x_1) > d$
	and thus $\mu(x_1) = \sigma(x_8) = b$. This means that $x_5$ does not have a unique maximum color in its open neighborhood,
	a contradiction.

	\medskip
	\noindent \textbf{Case 2:}
	Suppose now that $x = x_9$, i.e., $\sigma(x_9) = 5$. \quad
	Let $\sigma(x_4) = d$ and consequently, without loss of generality,
	we may assume that in $N(x_4)$ each of the colors $a$ and $b$ appears twice, and color $c$ appears once.
	Therefore $c > a$ and $c > b$.
	Observe that $\sigma(x_1) \neq c$, otherwise $\sigma(x_2) = \sigma(x_3) = d$, which is not possible as they are adjacent.
	So, without loss of generality, we may assume that $\sigma(x_1) = \sigma(x_5) = a$, 
	$\sigma(x_8) = b$, $\sigma(x_3) = d$, and $\sigma(x_2) = c$.
	Now, if $\sigma(x_7) = c$, then $x_1$ does not have a unique maximum color in its open neighborhood, since $b < c$.
	Therefore, $\sigma(x_7) = b$ and $\sigma(x_6) = c$.	
	Then, $\mu(x_7) = c$, meaning that $c > d$. 	
	This means that $x_9$ does not have a unique maximum color in its open neighborhood,
	a contradiction.
	
	\medskip
	It is easy to observe that $G$ admits a pUMo-coloring with $6$ colors.
\end{proof}

\section{Proper unique-maximum coloring with respect to closed neighborhood}
	
For this invariant, we also establish a tight upper bound for the outerplanar graphs.
From Theorem~\ref{thm:pumo-out}, using Proposition~\ref{prop:cl-op}, 
we immediately infer the following.
	
\begin{corollary}
	\label{cor:pumc-out}
	For every outerplanar graph $G$, it holds 
	$$
		\pumc(G) \le 5 \,.
	$$
\end{corollary}

The upper bound from Corollary~\ref{cor:pumc-out} is the best possible due to, e.g., the outerplanar graph $O_{\mathrm{pUMc}}$ depicted in Figure~\ref{fig:pumc-out5}.
\begin{figure}[htp!]
	$$
		\includegraphics{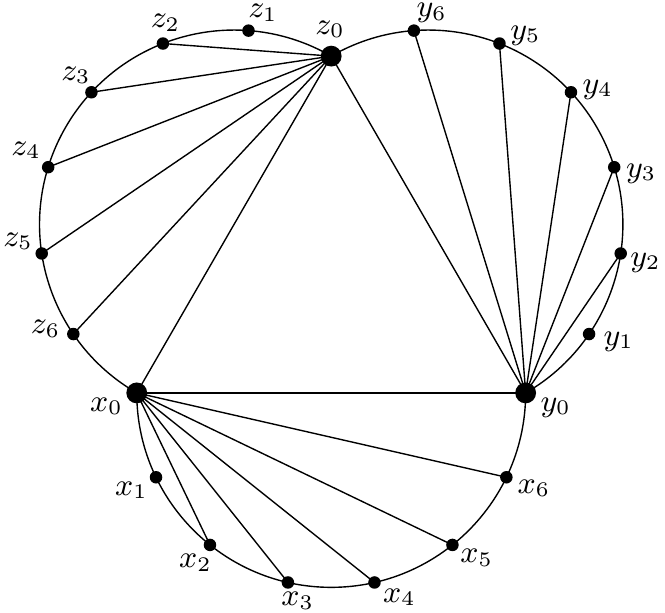}
	$$
	\caption{The outerplanar graph $O_{\mathrm{pUMc}}$ with $\pumc(O_{\mathrm{pUMc}}) = 5$.}
    \label{fig:pumc-out5}	
\end{figure}

\begin{proposition}
	\label{prop:pumc-out}
	The outerplanar graph $O_{\mathrm{pUMc}}$ in Figure~\ref{fig:pumc-out5} has $\pumc(G) = 5$.
\end{proposition}

\begin{proof}
	Let the vertices of $O_{\mathrm{pUMc}}$ be labeled as in Figure~\ref{fig:pumc-out5}.
	Suppose to the contrary that $O_{\mathrm{pUMc}}$ admits a pUMc-coloring $\sigma$ using at most $4$ colors.
	By the cyclic symmetry, we may assume that $\sigma(x_0) = a \in \set{1,2}$ and $\sigma(y_0) \in \set{3,4}$.
	Let $x_7 = y_0$ and $b \in \set{1,2}\setminus \set{a}$.
	Note that at most one of the vertices $x_i$, $i \in \set{1,\dots,7}$, is colored by $4$, 
	otherwise $x_0$ does not have a unique maximum color in its open neighborhood.
	The remaining vertices $x_i$ are colored with colors from $\set{b,3}$.
	Moreover, if $\sigma(x_j) = b$ for some $j \in \set{2,\dots,6}$,
	then $\mu(x_j) = 4$ and $\set{\sigma(x_{j-1}), \sigma(x_{j+1})} = \set{3,4}$.
	It follows that $\sigma(x_4) = 4$, otherwise $\sigma(x_2)=\sigma(x_3)= 3$ or $\sigma(x_5)=\sigma(x_6) =3$.
	Thus, $\sigma(y_0)=3$ and $\sigma(x_6)=b$, a contradiction.
\end{proof}

The bounds for planar graphs compared to the ones for outerplanar graphs are again higher as expected at first sight.
\begin{theorem}
	\label{thm:pumc}
	For every planar graph $G$, it holds 
	$$
		\pumc(G) \le 8\,.
	$$
\end{theorem}

We prove the theorem in a similar manner as Theorem~\ref{thm:pcfo}.
\begin{proof}
	By abusing the notation, we let $G$ represent also a fixed plane embedding of $G$.
	Color the vertices of $G$ properly by the Four Color Theorem with at most $4$ colors
	with an additional assumption that we use color $4$ on the maximal number of vertices.
	In this way, we obtain a coloring $\sigma$ in which every vertex is either colored with $4$ or it is adjacent to a vertex of color $4$.
	
	Let $V_1$ be the set of vertices of $G$ colored by $4$ and $V_2 = V(G)\setminus V_1$.	
	Now, we will use the facial closure $\Phi_G(V_2)$.
	By Theorem~\ref{thm:fac_um}, it admits a facial unique-maximum coloring $\alpha$ using at most $5$ colors;
	we use the set of colors $\set{4,5,6,7,8}$.
	
	Now, recolor the vertices of $G$ by setting
	$\sigma(v) = \alpha(v)$ for every $v \in V_1$.
	We claim that the coloring $\sigma$ is a pUMc-coloring of $G$.
	Obviously $\alpha$ is a proper coloring.
	Next, for every vertex $v$ in $V_2$,
	there is a unique maximal color from $\set{4,\dots,8}$ in its neighborhood, 
	since it is adjacent to at least one vertex from $V_1$.
	The cases with $d(v) \le 2$ are again trivial, and the case $d(v) \ge 3$ follows from the properties of $\Phi_G(V_2)$.
	For every vertex $v$ in $V_1$, all its neighbors have a color lower than itself, and thus $\mu(v) = \sigma(v)$.
\end{proof}

On the other hand, for some planar graphs at least $6$ colors are needed for a pUMo-coloring.
We define the graph $H_{\mathrm{pUMo}}$ as the planar graph obtained from the graph $K_4$ with every edge 
directed in such a way that the outdegree and indegree of every vertex is positive (see the left graph in Figure~\ref{fig:pumo-sub}). 
Then, we replace every arc of $K_4$ with a copy of the configuration $H$ (the right graph in Figure~\ref{fig:pumo-sub}) 
such that $x$ is identified with the initial vertex and $z$ is identified with the terminal vertex of an arc.
We denote the vertices of $H_{\mathrm{pUMo}}$ corresponding to the vertices of $K_4$ by $x_1$, $x_2$, $x_3$, and $x_4$,
and we refer to them as the {\em base vertices}.

\begin{figure}[htp!]
	$$
	\includegraphics{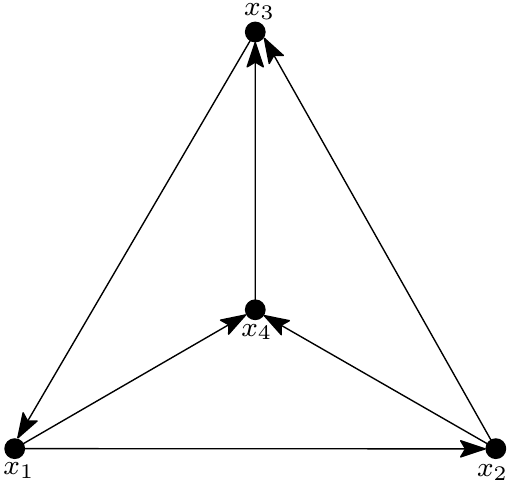} \quad\quad\quad\quad\quad
	\includegraphics{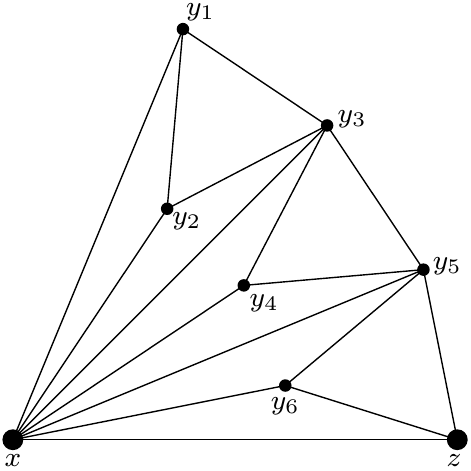}
	$$
	\caption{The graph $K_4$ with directed edges (left), and
		the configuration $H$ (right) by which we replace every arc of $K_4$.}
	\label{fig:pumo-sub}
\end{figure}

\begin{proposition}
	\label{prop:pumc}
	The planar graph $H_{\mathrm{pUMo}}$ has $\pumc(H_{\mathrm{pUMo}}) = 6$.
\end{proposition}

\begin{proof}
	Suppose to the contrary that $H_{\mathrm{pUMo}}$ admits a pUMo-coloring $\sigma$ with at most $5$ colors.
	If some base vertex is colored with $5$, then no other vertex has color $5$,
	otherwise at least one vertex of $H_{\mathrm{pUMo}}$ would be adjacent to two vertices of maximum color $5$.
	On the other hand, if no base vertex is colored with $5$, 
	then color $5$ is used on at most four copies of $H$;
	namely, every base vertex can be the initial vertex of at most one arc that corresponds to a copy of $H$ having a vertex of color $5$.	
	
	Therefore, in both cases above, there are at least two copies of $H$ not incident with a vertex of color $5$.	
	It follows that there is a copy of $H$ with $\sigma(x) \in \set{1,2,3}$.
	Let $\set{a,b,c} = \set{1,2,3}$.
	Set $\sigma(x) = a$.
	Then, $\set{\sigma(y_1),\sigma(y_2),\sigma(y_3)} = \set{b,c,4}$
	and $\set{\sigma(y_5),\sigma(y_6),\sigma(z)} = \set{b,c,4}$.
	This means that either $\sigma(y_i) = 4$ or $\mu(y_i) = 4$, for $i \in \set{3,5}$.
	Moreover, if $\sigma(y_3)=4$ (resp., $\sigma(y_5)=4$), 
	then $y_5$ (resp., $y_3$) has in its neighborhood maximum color $4$ twice, a contradiction.
	Therefore, $\set{\sigma(y_3),\sigma(y_5)}=\set{b,c}$ and consequently $\sigma(y_4) = 4$.
	But in this way, $y_3$ and $y_5$ both have the maximum color $4$ twice in their neighborhoods, a contradiction.

	We leave to the reader the exercise of confirming that $6$ colors are sufficient to color $H_{\mathrm{pUMo}}$.
\end{proof}

\section{Conclusion}	
\label{sec:con}
	
We began this paper by presenting results for the improper variants of considered colorings.
As already mentioned, there are still gaps between the lower and upper bounds in all the cases of Theorem~\ref{thm:im-pla}
and in three cases of Theorem~\ref{thm:im-out}.
Based on our experience with provided constructions, we believe that the current lower bounds are also correct and thus we propose the following.
\begin{conjecture}
	\label{conj:im-pla}
	For the classes of planar graphs $\mathcal{P}$ and outerplanar graphs $\mathcal{O}$, it holds
	\begin{itemize}
		\item[$(a)$] $\cf(\mathcal{P}) = 4$;
		\item[$(b)$] $\cfc(\mathcal{P}) = 3$;
		\item[$(c)$] $\um(\mathcal{P}) = 5$;
		\item[$(d)$] $\umc(\mathcal{P}) = 4$;
		\item[$(e)$] $\cf(\mathcal{O}) = 3$;
		\item[$(f)$] $\um(\mathcal{O}) = 4$;
		\item[$(g)$] $\umc(\mathcal{O}) = 3$;		
	\end{itemize}
\end{conjecture}

In the proper setting, while we established tight bounds for outerplanar graphs in all cases,
there are three gaps between the lower and the upper bounds for the corresponding chromatic numbers for planar graphs (see Theorem~\ref{thm:pr-pla}).
In two cases, we believe that the current lower bounds are also correct.
\begin{conjecture}
	For the class of planar graphs $\mathcal{P}$, it holds
	\begin{itemize}
		\item[$(a)$] $\pcf(\mathcal{P}) = 6$;
		\item[$(b)$] $\pumc(\mathcal{P}) = 6$.
	\end{itemize}
\end{conjecture}

We are particularly intrigued by the upper bound of $10$ for $\pum$.
It seems that the correct upper bound is much lower, 
but a simple example of the Fritsch graph needing $6$ colors
does not allow us to believe that there are no planar graphs $G$ with $\pum(G) = 7$.
Therefore, we conjecture the following.
\begin{conjecture}
	For the class of planar graphs $\mathcal{P}$, it holds
	$$
		\pum(\mathcal{P}) \le 7\,.
	$$
\end{conjecture}

Very recently, a relaxed version of pCFo-coloring was considered by Petru\v{s}evski and \v{S}kre\-kovski~\cite{PetSkr21}.
It is called an {\em odd coloring} (with respect to open neighborhood); in our terms, denote it a {\em $\mathrm{pODDo}$-coloring}. 
It is defined as a proper coloring where in the open neighborhood of every vertex there is a color appearing an odd number of times.
%
Among other results Petru\v{s}evski and \v{S}krekovski proved that $9$ colors suffice for an pODDo-coloring of any planar graph
and conjectured that the bound can be reduced to $5$.

In a very short interval of less than a month, three additional papers considering this coloring were published on arXiv.
First, Cranston~\cite{Cra22} established several results for sparse graphs, and as a corollary, he obtained bounds $6$ and $5$
for planar graphs of girth $6$ and at least $7$, respectively.
Then, Caro et al.~\cite{CarPetSkr22} established the bound $8$ for planar graphs with specific properties, 
and finally, Petr and Portier~\cite{PetPor22} proved the bound $8$ for all planar graphs.

The same bound follows also from our bound on pCFo-coloring of planar graphs in Theorem~\ref{thm:pcfo}; 
namely, our coloring has a stronger property that there is a unique color in every open neighborhood.
The conjectured lower bound of $5$ colors for pODDo-coloring of planar graphs is still widely open.

\paragraph{Acknowledgment.} 
The authors thank the anonymous reviewers for their comments and pointing out additional references on the considered topics.
I. Fabrici, S. Rindo\v{s}ov\'{a}, and R. Sot\'{a}k acknowledge the financial support from the projects APVV--19--0153 and VEGA 1/0574/21.
B. Lu\v{z}ar was partially supported by the Slovenian Research Agency Program P1--0383 and the projects J1--1692 and J1--3002.


\end{document}